\documentclass[a4paper, 11pt]{article}

\usepackage{amsmath, amssymb, amsthm} % AMS series
\usepackage{color}
\usepackage[top=40truemm,bottom=30truemm,left=30truemm,right=30truemm]{geometry}
\usepackage{hyperref}
\usepackage[all]{xy}

\newcommand{\R}{\mathbb{R}}

\DeclareMathOperator{\lip}{lip}

\DeclareMathOperator{\Map}{Map}

\theoremstyle{plain}
\newtheorem{theorem}{Theorem}[section]
\newtheorem{lemma}[theorem]{Lemma}
\newtheorem{proposition}[theorem]{Proposition}
\newtheorem{corollary}[theorem]{Corollary}

\theoremstyle{definition}
\newtheorem{definition}[theorem]{Definition}

\theoremstyle{remark}
\newtheorem{remark}[theorem]{Remark}

\makeatletter

\@addtoreset{equation}{section}
\makeatother

\begin{document}

\title{Some discontinuous functional differential equation
and its connection to smoothness of composition operators in $L^p$}
\author{Junya Nishiguchi\thanks{Mathematical Science Group, Advanced Institute for Materials Research, Tohoku University,
2-1-1 Katahira, Aoba-ku, Sendai, 980-8577, Japan}
\footnote{E-mail: \url{junya.nishiguchi.b1@tohoku.ac.jp}}}
\date{}

\maketitle

\begin{abstract}
The objective of this paper is to deepen the understanding of the connection
between the continuous and smooth dependence of solutions on initial conditions
and the regularity of the history functionals for retarded functional differential equations.
We consider some differential equation with a single constant delay with the history space of $L^p$-type
and obtain the above dependence result by assuming the growth rate of the nonlinearity and its derivative.
The corresponding history functional is discontinuous,
and it becomes clear that there are the continuity and the smoothness
of the composition operators (also called the superposition operators or the Nemytskij operators between $L^p$-spaces
behind the dependence results.

\begin{flushleft}
\end{flushleft}

\begin{flushleft}
\end{flushleft}

\end{abstract}

\tableofcontents

\section{Introduction}

Delay differential equations (DDEs) are mathematically formulated as 
retarded functional differential equations (RFDEs) (see \cite{Hale 1963b}).
For a given RFDE, the functional (called the \textit{history functional} in this paper)
assigning the history of the unknown function $x$ at time $t$ to the derivative $\dot{x}(t)$ constitutes its main part.
It is usual to choose the Banach space $C([-R, 0], \R^N)$ of continuous functions as the space of initial histories,
where $R > 0$ is the maximal delay and $N \ge 1$ is an integer.
However, other choices are also possible.
An example is the quotient normed space $M^p([-R, 0], \R^N)$
of a seminormed space $\mathcal{L}^p([-R, 0], \R^N)$ endowed with the seminorm defined by
	\begin{equation}\label{eq:seminorm}
		\|\phi\|
		= \bigl( \|\phi\|_{L^p}^p + |\phi(0)|^p \bigr)^{\frac{1}{p}}
		= \left( \int_{-R}^0 |\phi(\theta)|^p \mspace{2mu} \mathrm{d}\theta + |\phi(0)|^p \right)^{\frac{1}{p}}.
	\end{equation}
Here $\mathcal{L}^p([-R, 0], \R^N)$ denotes the linear space of $p$-th power Lebesgue integrable functions
from $[-R, 0]$ to $\R^N$, and $|\cdot|$ is a norm on $\R^N$.
The quotient normed space is a Hilbert space for $p = 2$, which is advantageous.
We refer the reader to \cite{Delfour--Mitter 1972c} for the detail.
We also refer the reader to \cite{Hale--Lunel 1993} for a general reference of the theory of RFDEs.

A difficulty arising from the choice of the above space is the discontinuity
of the history functionals which corresponds to differential equations with constant delay.
For the simplicity, we consider a delay differential equation
	\begin{equation}\label{eq:single const delay I}
		\dot{x}(t) = f(x(t - r)),
	\end{equation}
where $f \colon \R^N \to \R^N$ is at least continuous and $0 < r \le R$ is a parameter.
Then the history functional $F$ is given by
	\begin{equation*}
		F(\phi, r) = f(\phi(-r)),
	\end{equation*}
however, it is discontinuous with respect to the seminorm $\|\cdot\|$ given in \eqref{eq:seminorm}.
See also \cite{Hale 1969} for the discussion about this discontinuity problem.

The objective of this paper is to obtain the continuous and smooth dependence
of the solution of \eqref{eq:single const delay I} on the initial conditions
in order to deepen the understanding of the above mentioned discontinuity problem.
For this purpose, we consider an initial value problem (IVP)
	\begin{equation}\label{eq:IVP, single const delay I, discontinuous histories}
		\left\{
		\begin{alignedat}{2}
			\dot{x}(t) &= f(x(t - r)), & \mspace{20mu} & t \ge 0, \\
			x(t) &= \phi(t), & & t \in [-R, 0]
		\end{alignedat}
		\right.
	\end{equation}
for each $\phi \in \mathcal{L}^1([-R, 0], \R^N)$ and each $r \in (0, R]$.
By using the corresponding history functional, this IVP can be written as
	\begin{equation*}
		\left\{
		\begin{alignedat}{2}
			\dot{x}(t) &= F(R_tx, r), & \mspace{20mu} & t \ge 0, \\
			R_0x &= \phi,
		\end{alignedat}
		\right.
	\end{equation*}
where
	\begin{equation*}
		R_tx \colon [-R, 0] \ni \theta \mapsto x(t + \theta) \in \R^N
	\end{equation*}
denotes the \textit{history} of $x$ at $t$ with the past interval $[-R, 0]$.
We note that the solution $x(\cdot; \phi, r)$ of \eqref{eq:IVP, single const delay I, discontinuous histories}
is given by
	\begin{equation*}
		x(t; \phi, r) = \phi(0) + \int_0^t f(\phi(s - r)) \mspace{2mu} \mathrm{d}s
	\end{equation*}
on the interval $[0, r]$, which is continued to $[-R, +\infty)$ by the method of steps.

We briefly review the previous results about the problem of continuous and smooth dependence.
For RFDEs with history space $C([-R, 0], \R^N)$ and with continuous (resp.\ smooth) history functionals,
the continuous (resp.\ smooth) dependence is a classical result.
For RFDEs with $M^p([-R, 0], \R^N)$,
a general theory of the existence, uniqueness, and continuous dependence is discussed
in \cite{Herdman--Burns 1979} and \cite{Kappel--Schappacher 1978} with the necessary hypotheses of history functionals.
See \cite{Webb 1976} for the treatment of RFDEs as nonlinear semigroups in $M^p([-R, 0], \R^N)$,
where the Lipschitz continuity of the history functional is a basic assumption.
To the best of the author's knowledge, there are less result about the smooth dependence in our setting.

It seems difficult to obtain the continuous and smooth dependence of the solution to \eqref{eq:single const delay I}
for a general function $f \colon \R^N \to \R^N$ because the corresponding history functional $F$ is discontinuous.
Therefore, it will be reasonable to restrict the behavior of $f(x)$ as $|x| \to \infty$.
The assumption used for the continuous dependence is a condition that
$|f(x)| = O(|x|^\alpha)$ as $|x| \to \infty$ for some $\alpha \ge 1$,
under which the continuous dependence can be proved with the appropriate exponent $p = \alpha$ of the history space.
For the smooth dependence, it is natural to assume the continuous differentiability of $f$.
By assuming the behavior of the derivative $Df \colon \R^N \to M_N(\R)$,
where $M_N(\R)$ denotes the set of real $N \times N$ matrices,
the smooth dependence can be obtained.
Here $\|Df(x)\| = O(|x|^\alpha)$ as $|x| \to \infty$ is the assumption,
whose use is motivated by the discussion by Kappel \& Schappacher~\cite{Kappel--Schappacher 1978}.
There are the continuity and the smoothness of the composition operators
(also called the superposition operator or the Nemytskij operator) in $L^p$ behind the these results.
We refer the reader to \cite{Goldberg--Kampowsky--Troltzsch 1992}
as a reference of the continuity and smoothness of the composition operators in function spaces between $L^p$-spaces.

This paper is organized as follows.
In Section~\ref{sec:history space of L^p-type},
we introduce the notions about general history spaces
which are fundamentally used in \cite{Nishiguchi 2017} and \cite{Nishiguchi preprint}.
Furthermore, we define the above mentioned history space of $L^p$-type and discuss its fundamental properties.
In Section~\ref{sec:main results}, we prove the main results of this paper,
which consist of the continuous dependence (Theorem~\ref{thm:continuous dependence, discontinuous history functional}),
the smooth dependence (Thorem~\ref{thm:smooth dependence, discontinuous history functional}),
and the regularity of solution semiflows
(Theorems~\ref{thm:continuous semiflow, discontinuous history functional}
and \ref{thm:C^1-semiflow, discontinuous history functional}).
Here the continuity and smoothness of the composition operators in $L^p$
(Theorems~\ref{thm:continuity of composition operators in L^p} and \ref{thm:smoothness of composition operators in L^p})
are fundamental and are proved in Appendix~\ref{sec:composition operators in L^p}
to keep this paper self-contained.
The regularity of maximal semiflows is discussed in Appendix~\ref{sec:regularity of maximal semiflows}.

\section{Preliminary: History space of $L^p$-type}\label{sec:history space of L^p-type}

Let $R > 0$ be a constant and $N \ge 1$ be an integer.
The linear space of all maps from $[-R, 0]$ to $\R^N$ is denoted by $\Map([-R, 0], \R^N)$.
Let $\R_+$ denote the set of all nonnegative real numbers.

\begin{definition}[History space]
A linear subspace $H \subset \Map([-R, 0], \R^N)$ is called a \textit{history space}
with the past interval $[-R, 0]$
if the topology of $H$ is given so that the linear operations on $H$ are continuous.
\end{definition}

\begin{definition}[Prolongation]
Let $(t_0, \phi_0) \in \R \times \Map([-R, 0], \R^N)$ be given.
For each left-closed interval $J \subset \R$ with the left-end point $t_0$,
a function $\gamma \colon J + [-R, 0] \to \R^N$ is called a \textit{prolongation} of $(t_0, \phi_0)$
if the restriction $\gamma|_J \colon J \to \R^N$ is continuous and $R_{t_0}\gamma = \phi_0$.
When $t_0 = 0$, it is simply called a prolongation of $\phi_0$.
\end{definition}

\begin{definition}[Static prolongation]
Let $\phi \in \Map([-R, 0], \R^N)$.
The function $\bar{\phi} \colon \R_+ + [-R, 0] \to \R^N$ defined by
	\begin{equation*}
		\bar{\phi}(t) =
		\begin{cases}
			\phi(t) & (t \in [-R, 0]), \\
			\phi(0) & (t \in \R_+)
		\end{cases}
	\end{equation*}
is called the \textit{static prolongation} of $\phi$.
\end{definition}

\begin{definition}[Prolongable history space]\label{dfn:prolongable history space}
A history space $H \subset \Map([-R, 0], \R^N)$ is said to be \textit{prolongable}
if the following property is satisfied:
For every $\phi \in H$ and every prolongation $\gamma \colon J + [-R, 0] \to \R^N$ of $\phi$,
	\begin{equation*}
		J \ni t \mapsto R_t\gamma \in H
	\end{equation*}
is a well-defined continuous map.
When the above map fails to be continuous, $H$ is said to be \textit{closed under prolongations}.
\end{definition}

\begin{definition}[Regulation by prolongations]\label{dfn:regulation by prolongations}
Let $H \subset \Map([-R, 0], \R^N)$ be a history space which is closed under prolongations.
It is said that $H$ is \textit{regulated by prolongations}
if the inclusion $C([-R, 0], \R^N) \subset H$ is continuous.
\end{definition}

\begin{remark}
This notion is introduced in \cite{Nishiguchi preprint} permitting the infinite past interval $(-\infty, 0]$.
Indeed, the condition in Definition~\ref{dfn:regulation by prolongations} is a necessary and sufficient condition
for the regulation by prolongations when the past interval is $[-R, 0]$.
See also \cite{Nishiguchi 2017}.
\end{remark}

\begin{definition}[History space of $L^p$-type]
Let $a < b$ be real numbers and $1 \le p < \infty$.
For each Lebesgue measurable function $\phi \colon [a, b] \to \R^N$, let
	\begin{equation*}
		\|\phi\|_{\bar{\mathcal{L}}^p[a, b]}
		:= \bigl( \|\phi\|_{L^p[a, b]}^p + |\phi(b)|^p \bigr)^{\frac{1}{p}}.
	\end{equation*}
Let
	\begin{equation*}
		\bar{\mathcal{L}}^p([a, b], \R^N)
		:= \bigl( \mathcal{L}^p([a, b], \R^N), \|\cdot\|_{\bar{\mathcal{L}}^p[a, b]} \bigr).
	\end{equation*}
The history space $\bar{\mathcal{L}}^p([-R, 0], \R^N)$ is said to be of \textit{$L^p$-type}.
\end{definition}

\begin{remark}
$\bar{\mathcal{L}}^p([a, b], \R^N)$ is a seminormed space.
The associated equivalence class of $\phi$ is given by
	\begin{equation*}
		[\phi]
		= \{\mspace{2mu} \psi \in \mathcal{L}^p([a, b], \R^N) :
		\|\psi - \phi\|_{\bar{\mathcal{L}}^p[a, b]} = 0 \mspace{2mu}\},
	\end{equation*}
which is equal to
	\begin{equation*}
		\{\mspace{2mu} \psi \in \mathcal{L}^p([a, b], \R^N) :
		\text{$\psi = \phi$ almost everywhere, and $\psi(0) = \phi(0)$} \mspace{2mu}\}.
	\end{equation*}
\end{remark}

The following is stated in \cite[Proposition 2.1]{Delfour--Mitter 1972c}.

\begin{lemma}
For real numbers $a < b$ and $1 \le p < \infty$, let
	\begin{align*}
		\bar{L}^p([a, b], \R^N)
		&:= \bigl( L^p([a, b], \R^N) \times \R^N, \|\cdot\|_{\bar{L}^p[a, b]} \bigr), \\
		\|([\phi]_\mathrm{a.e.}, \eta)\|_{\bar{L}^p[a, b]}
		&:= \bigl( \|[\phi]_\mathrm{a.e.}\|_{L^p[a, b]}^p + |\eta|^p \bigr)^{\frac{1}{p}},
	\end{align*}
where
	\begin{equation*}
		[\phi]_\mathrm{a.e.}
		:= \{\mspace{2mu} \psi : \text{$\psi = \phi$ almost everywhere} \mspace{2mu}\}.
	\end{equation*}
Then the quotient normed space
	\begin{equation*}
		M^p([a, b], \R^N) := \mathcal{L}^p([a, b], \R^N)/\|\cdot\|_{\bar{\mathcal{L}}^p[a, b]}
	\end{equation*}
is isometrically isomorphic to $\bar{L}^p([a, b], \R^N)$.
\end{lemma}

\begin{proof}
We define a map $T \colon \bar{L}^p([a, b], \R^N) \to M^p([a, b], \R^N)$ by
	\begin{equation*}
		T([\phi]_\mathrm{a.e.}, \eta) = [\phi^\eta],
	\end{equation*}
where the map $\phi^\eta \colon [a, b] \to \R^N$ is defined by
	\begin{equation*}
		\phi^\eta(t) =
		\begin{cases}
			\phi(t) & (t \in [a, b)), \\
			\eta & (t = b).
		\end{cases}
	\end{equation*}
Then $T$ is well-defined and
	\begin{align*}
		\|T([\phi]_\mathrm{a.e.}, \eta)\|
		&= \|\phi^\eta\|_{\bar{\mathcal{L}}^p[a, b]} \\
		&= \bigl( \|\phi\|_{L^p[a, b]}^p + |\eta|^p \bigr)^{\frac{1}{p}} \\
		&= \|([\phi]_\mathrm{a.e.}, \eta)\|_{\bar{L}^p[a, b]}.
	\end{align*}
This completes the proof.
\end{proof}

\begin{remark}
The quotient normed space $M^p([a, b], \R^N)$ is identified with the normed space $\bar{L}^p([a, b], \R^N)$ if desired.
\end{remark}

\begin{remark}
The Banach space $\bar{L}^p([-R, 0], \R^N)$ is used by many authors, e.g., Delfour \& Mitter~\cite{Delfour--Mitter 1972c},
Webb~\cite{Webb 1976}, Kappel \& Schappacher~\cite{Kappel--Schappacher 1978}, Delfour~\cite{Delfour 1980},
Burns, Herdman, \& Stech~\cite{Burns--Herdman--Stech 1983}, Breda~\cite{Breda 2010}, and
Chekroun et al.~\cite{Chekroun--Ghil--Liu--Wang 2016}.
When the delay is infinite, a Banach space which is similar to the above space  
is used by Herdman \& Burns~\cite{Herdman--Burns 1979}.
We refer the reader to Hino, Murakami, \& Naito~\cite{Hino--Murakami--Naito 1991}
for a general reference of functional differential equations with infinite delay. 
\end{remark}

\begin{lemma}\label{lem:prolongability of L^p type history space}
Let $1 \le p < \infty$ and $T > 0$.
Let $x \colon [-R, T] \to \R^N$ be a map satisfying $R_0x \in \bar{\mathcal{L}}^p([-R, 0], \R^N)$.
If the restriction $x|_{[0, T]} \colon [0, T] \to \R^N$ is continuous, then the map
	\begin{equation*}
		[0, T] \ni t \mapsto R_tx \in \bar{\mathcal{L}}^p([-R, 0], \R^N)
	\end{equation*}
is continuous.
\end{lemma}

\begin{proof}
Let $t_0 \in [0, T]$.
Then for all $t \in [0, T]$, we have
	\begin{equation*}
		\|R_tx - R_{t_0}x\|_{\bar{\mathcal{L}}^p[-R, 0]}^p
		= \int_{-R}^0 |x(t + \theta) - x(t_0 + \theta)|^p \mspace{2mu} \mathrm{d}\theta + |x(t) - x(t_0)|^p,
	\end{equation*}
where the right-hand side converges to $0$ as $t \to t_0$.
\end{proof}

\begin{remark}
Lemma~\ref{lem:prolongability of L^p type history space} means that
the history space $\bar{\mathcal{L}}^p([-R, 0], \R^N)$ is prolongable.
\end{remark}

\begin{lemma}\label{lem:regulation by prolongations of L^p type history space}
Let $a < b$ be real numbers and $1 \le p < \infty$.
Then for all $x \in C([a, b], \R^N)$,
	\begin{equation*}
		\|x\|_{\bar{\mathcal{L}}^p[a, b]} \le (b - a + 1)^{\frac{1}{p}}\|x\|_{C[a, b]},
	\end{equation*}
holds.
\end{lemma}

\begin{proof}
Let $x \in C([a, b], \R^N)$.
Then we have
	\begin{align*}
		\|x\|_{\bar{\mathcal{L}}^p[a, b]}^p
		= \int_a^b |x(t)|^p \mspace{2mu} \mathrm{d}t + |x(b)|^p
		\le (b - a + 1)\|x\|_{C[a, b]}^p,
	\end{align*}
from which the inequality is obtained.
\end{proof}

\begin{remark}
Lemma~\ref{lem:regulation by prolongations of L^p type history space} shows that
the inclusion $C([a, b], \R^N) \subset \bar{\mathcal{L}}^p([a, b], \R^N)$ is continuous.
This means that the history space $\bar{\mathcal{L}}^p([-R, 0], \R^N)$ of $L^p$-type is
regulated by prolongations.
\end{remark}

\section{Main results}\label{sec:main results}

In the proofs,
the function space $\bar{\mathcal{L}}^p([-R, 0], \R^N)$ is abbreviated as $\bar{\mathcal{L}}^p[-R, 0]$.
This is similar to other function spaces.

\subsection{Continuous dependence}

\begin{lemma}\label{lem:continuity of prolongation operator in L^p type history space}
Let $1 \le p < \infty$ and $T > 0$.
Then for all $\phi \in \bar{\mathcal{L}}^p([-R, 0], \R^N)$,
	\begin{equation*}
		\|\bar{\phi}\|_{\bar{\mathcal{L}}^p[-R, T]} \le (1 + T)^{\frac{1}{p}}\|\phi\|_{\bar{\mathcal{L}}^p[-R, 0]}
	\end{equation*}
holds.
\end{lemma}

\begin{proof}
Let $\phi \in \bar{\mathcal{L}}^p[-R, 0]$.
Then
	\begin{align*}
		\int_{-R}^T |\bar{\phi}(\theta)|^p \mspace{2mu} \mathrm{d}\theta + |\bar{\phi}(T)|^p
		&= \int_{-R}^0 |\bar{\phi}(\theta)|^p \mspace{2mu} \mathrm{d}\theta + (1 + T)|\phi(0)|^p \\
		&\le (1 + T)(\|\phi\|_{L^p[-R, 0]}^p + |\phi(0)|^p),
	\end{align*}
from which the inequality is obtained.
\end{proof}

\begin{remark}
Lemma~\ref{lem:continuity of prolongation operator in L^p type history space} shows that
the \textit{prolongation operator}
	\begin{equation*}
		\bar{\mathcal{L}}^p([-R, 0], \R^N) \ni \phi \mapsto \bar{\phi} \in \bar{\mathcal{L}}^p([-R, T], \R^N)
	\end{equation*}
is continuous.
This is not continuous with respect to $\|\cdot\|_{L^p[-R, 0]}$-norm.
\end{remark}

\begin{theorem}\label{thm:continuous dependence, discontinuous history functional}
Let $0 < T < R$ be given.
We decompose the solution $x(\cdot; \phi, r)$ by
	\begin{equation*}
		x(\cdot; \phi, r) = y(\cdot; \phi, r) + \bar{\phi}.
	\end{equation*}
Suppose $|f(x)| = O(|x|^\alpha)$ as $|x| \to \infty$ for some $\alpha \ge 1$.
Then
	\begin{equation*}
		\mathcal{L}^\alpha([-R, 0], \R^N) \ni \phi \mapsto y(\cdot; \phi, r) \in C([-R, T], \R^N)
	\end{equation*}
is continuous uniformly  in $r \in [T, R]$.
\end{theorem}

\begin{proof}
Let $r \in [T, R]$ be a parameter and $\phi_0 \in \mathcal{L}^\alpha[-R, 0]$.
Then for all $\phi \in \mathcal{L}^\alpha[-R, 0]$, we have
	\begin{align*}
		\|y(\cdot; \phi, r) - y(\cdot; \phi_0, r)\|_{C[-R, T]}
		&\le \int_0^T |f(\phi(t - r)) - f(\phi_0(t - r))| \mspace{2mu} \mathrm{d}t \\
		&\le \int_{-R}^0 |f(\phi(\theta)) - f(\phi_0(\theta))| \mspace{2mu} \mathrm{d}\theta \\
		&= \|f \circ \phi - f \circ \phi_0\|_{L^1[-R, 0]}.
	\end{align*}
Here the last term converges to $0$ as $\|\phi - \phi_0\|_{L^\alpha[-R, 0]} \to 0$
by the continuity of composition operators
(see Theorem~\ref{thm:continuity of composition operators in L^p}).
This show the continuity.
\end{proof}

\begin{corollary}\label{cor:continuous dependence, discontinuous history functional}
Let $0 < T < R$ be given.
Suppose $|f(x)| = O(|x|^\alpha)$ as $|x| \to \infty$ for some $\alpha \ge 1$.
Then for all $\alpha \le p < \infty$,
	\begin{equation*}
		\bar{\mathcal{L}}^p([-R, 0], \R^N) \ni \phi \mapsto x(\cdot; \phi, r) \in \bar{\mathcal{L}}^p([-R, T], \R^N)
	\end{equation*}
is continuous uniformly in $r \in [T, R]$.
\end{corollary}

\begin{proof}
Let $\alpha \le p < \infty$ and $r \in [T, R]$ be a parameter.
We use the decomposition
	\begin{equation*}
		x(\cdot; \phi, r) = y(\cdot; \phi, r) + \bar{\phi}.
	\end{equation*}
Lemma~\ref{lem:continuity of prolongation operator in L^p type history space} states that
the prolongation operator
	\begin{equation*}
		\bar{\mathcal{L}}^p[-R, 0] \ni \phi \mapsto \bar{\phi} \in \bar{\mathcal{L}}^p[-R, T]
	\end{equation*}
is continuous.
Furthermore, the inclusions
	\begin{equation*}
		\bar{\mathcal{L}}^p[-R, 0] \subset \mathcal{L}^\alpha[-R, 0],
			\mspace{15mu}
		C[-R, T] \subset \bar{\mathcal{L}}^p[-R, T]
	\end{equation*}
are continuous.
Therefore, the continuity follows by Theorem~\ref{thm:continuous dependence, discontinuous history functional}.
\end{proof}

\begin{remark}
When $f$ is bounded, $|f(x)| = O(|x|)$ as $|x| \to \infty$.
\end{remark}

When $f$ is Lipschitz continuous,
we can obtain a stronger result without the continuity of composition operators in $L^p$.

\begin{proposition}
Let $0 < T < R$ be given.
If $f$ is Lipschitz continuous,
then the map
	\begin{equation*}
		\bar{\mathcal{L}}^1([-R, 0], \R^N) \ni \phi \mapsto x(\cdot; \phi, r) \in \bar{\mathcal{L}}^1([-R, T], \R^N)
	\end{equation*}
is Lipschitz continuous uniformly in $r \in [T, R]$.
\end{proposition}

\begin{proof}
We use the decomposition given by
	\begin{equation*}
		x(\cdot; \phi, r) = y(\cdot; \phi, r) + \bar{\phi}.
	\end{equation*}
Then for each $\phi_1, \phi_2 \in \bar{\mathcal{L}}^1[-R, 0]$, we have
	\begin{equation*}
		\|\phi_1 - \phi_2\|_{\bar{\mathcal{L}}^1[-R, T]}
		\le (1 + T)\|\phi_1 - \phi_2\|_{\bar{\mathcal{L}}^1[-R, 0]}
	\end{equation*}
and
	\begin{align*}
		\|y(\cdot; \phi_1, r) - y(\cdot; \phi_2, r)\|_{C[-R, T]}
		&\le \int_0^T |f(\phi_1(s - r)) - f(\phi_2(s - r))| \mspace{2mu} \mathrm{d}s \\
		&\le \lip(f)T\|\phi_1 - \phi_2\|_{L^1[-R, 0]}.
	\end{align*}
By combining these inequalities and
by using the regulation by prolongations (Lemma~\ref{lem:regulation by prolongations of L^p type history space}),
we obtain
	\begin{align*}
		\|x(\cdot; \phi_1, r) - x(\cdot; \phi_2, r)\|_{\bar{\mathcal{L}}^1[-R, T]}
		&\le \|y(\cdot; \phi_1, r) - y(\cdot; \phi_2, r)\|_{\bar{\mathcal{L}}^1[-R, T]}
		+ \|\phi_1 - \phi_2\|_{\bar{\mathcal{L}}^1[-R, T]} \\
		&\le \left[ \frac{\lip(f)T}{T + R + 1} + (1 + T) \right] \|\phi_1 - \phi_2\|_{\bar{\mathcal{L}}^1[-R, 0]}.
	\end{align*}
This shows the conclusion.
\end{proof}

\subsection{Smooth dependence}

\begin{theorem}\label{thm:smooth dependence, discontinuous history functional}
Let $0 < T < R$ be given.
Suppose that $f$ is of class $C^1$ and $\|Df(x)\| = O(|x|^\alpha)$ as $|x| \to \infty$ for some $\alpha \ge 1$.
For each $(\phi, r) \in \mathcal{L}^{\alpha + 1}([-R, 0], \R^N) \times [T, R]$,
we define a linear map
	\begin{equation*}
		B_{\phi, r} \colon \mathcal{L}^{\alpha + 1}([-R, 0], \R^N) \to C([-R, T], \R^N)
	\end{equation*}
by
	\begin{equation*}
		B_{\phi, r}\chi(t) =
		\begin{cases}
			0 & (t \in [-R, 0]), \\
			\int_0^t Df(\phi(s - r))\chi(s - r) \mspace{2mu} \mathrm{d}s & (t \in [0, T]).
		\end{cases}
	\end{equation*}
Then $B_{\phi, r}$ is a bounded linear operator, and
	\begin{equation*}
		\|B_{\phi, r} - B_{\phi_0, r}\|
		\le \|Df \circ \phi - Df \circ \phi_0\|_{L^q[-R, 0]},
	\end{equation*}
where $q$ is the H\"{o}lder conjugate of $\alpha + 1$.
Furthermore, we have
	\begin{equation*}
		\|y(\cdot; \phi + \chi, r) - y(\cdot; \phi, r) - B_{\phi, r}\chi\|_{C[-R, T]}
		= o(\|\chi\|_{L^{\alpha + 1}[-R, 0]})
	\end{equation*}
as $\|\chi\|_{L^{\alpha + 1}[-R, 0]} \to 0$, where $y(\cdot; \phi, r)$ is defined by
	\begin{equation*}
		x(\cdot; \phi, r) = y(\cdot; \phi, r) + \bar{\phi}.
	\end{equation*}
\end{theorem}

\begin{proof}
Let $r \in [T, R]$ be a parameter and $q$ be the H\"{o}lder conjugate of $\alpha + 1$.

\textbf{Step 1. Order estimate of the Lipschitz constant}

We choose $C_1, C_2 > 0$ so that
	\begin{equation*}
		\|Df(x)\| \le C_1|x|^\alpha + C_2 \mspace{20mu} (\forall x \in \R^N).
	\end{equation*}
Let $M > 0$.
Then for all $x_1, x_2 \in \bar{B}(0; M)$ (the closed ball with center $0$ and radius $M$), we have
	\begin{equation*}
		|f(x_1) - f(x_2)| \le \sup_{x \in \bar{B}(0; M)} \|Df(x)\| \cdot |x_1 - x_2|.
	\end{equation*}
This implies
	\begin{equation*}
		\lip \bigl( f|_{\bar{B}(0; M)} \bigr) \le C_1M^\alpha + C_2 =: L(M).
	\end{equation*}

\textbf{Step 2. Boundedness of $B_{\phi, r}$}

Let $\phi \in \mathcal{L}^{\alpha + 1}[-R, 0]$.
Then for all $\chi \in \mathcal{L}^{\alpha + 1}[-R, 0]$, we have
	\begin{align*}
		\|B_{\phi, r}\chi\|_{C[-R, T]}
		&\le \int_0^T \|Df(\phi(t - r))\||\chi(t - r)| \mspace{2mu} \mathrm{d}t \\
		&\le \int_0^T \|Df(\phi(\theta))\||\chi(\theta)| \mspace{2mu} \mathrm{d}\theta \\
		&\le \|Df \circ \phi\|_{L^q[-R, 0]} \|\chi\|_{L^{\alpha + 1}[-R, 0]},
	\end{align*}
which implies
	\begin{equation*}
		\|B_{\phi, r}\| \le \|Df \circ \phi\|_{L^q[-R, 0]}.
	\end{equation*}
Since
	\begin{align*}
		\|Df(\phi(\theta))\|^q
		&\le (C_1|\phi(\theta)|^\alpha + C_2)^q \\
		&\le 2^{q - 1}(C_1^q|\phi(\theta)|^{\alpha q} + C_2^q)
	\end{align*}
for all $\theta \in [-R, 0]$ and
	\begin{equation*}
		\alpha q = \frac{\alpha(\alpha + 1)}{\alpha} = \alpha + 1,
	\end{equation*}
$\|Df \circ \phi\|_{L^q[-R, 0]} < \infty$.
Therefore, $B_{\phi, r} \colon \mathcal{L}^{\alpha + 1}[-R, 0] \to C[-R, T]$ is a bounded linear operator.

\textbf{Step 3. G\^{a}teaux differentiability}

Let $\phi, \chi \in \mathcal{L}^{\alpha + 1}[-R, 0]$ be given
and $(h_n)_{n = 1}^\infty$ be a sequence in $[-1, 1] \setminus \{0\}$ which converges to $0$.
Then we have
	\begin{align*}
		&\left\| \frac{1}{h_n} [y(\cdot; \phi + h_n\chi, r) - y(\cdot; \phi, r)] - B_{\phi, r}\chi \right\|_{C[-R, T]} \\
		&\le
		\int_0^T
			\left| \frac{1}{h_n} [f((\phi + h_n\chi)(s - r)) - f(\phi(s - r))] - Df(\phi(s - r))\chi(s - r) \right|
		\mspace{2mu} \mathrm{d}s \\
		&\le
		\int_{-R}^0
			\left| \frac{1}{h_n} [f((\phi + h_n\chi)(\theta)) - f(\phi(\theta))] - Df(\phi(\theta))\chi(\theta) \right|
		\mspace{2mu} \mathrm{d}\theta.
	\end{align*}
We define functions $g_n, g \colon [-R, 0] \to \R^N$ by
	\begin{equation*}
		g_n(\theta) = \frac{1}{h_n} [f((\phi + h_n\chi)(\theta)) - f(\phi_0(\theta))],
			\mspace{15mu}
		g(\theta) = Df(\phi(\theta))\chi(\theta)
	\end{equation*}
for all $\theta \in [-R, 0]$.
Then for each $\theta \in [-R, 0]$, $g_n(\theta) \to g(\theta)$ as $n \to \infty$, and
	\begin{align*}
		|g_n(\theta)|
		&= \frac{1}{|h_n|} |f((\phi + h_n\chi)(\theta)) - f(\phi(\theta))| \\
		&\le \frac{1}{|h_n|} L(|\phi(\theta)| + |\chi(\theta)|) |h_n\chi(\theta)| \\
		&= L(|\phi(\theta)| + |\chi(\theta)|) |\chi(\theta)|.
	\end{align*}
Since
	\begin{equation*}
		\int_{-R}^0 L(|\phi(\theta)| + |\chi(\theta)|) |\chi(\theta)| \mspace{2mu} \mathrm{d}\theta
		\le \left( \int_{-R}^0 L(|\phi(\theta)| + |\chi(\theta)|)^q \mspace{2mu} \mathrm{d}\theta \right)^{1/q}
		\|\chi\|_{L^{\alpha + 1}[-R, 0]},
	\end{equation*}
where
	\begin{align*}
		L(|\phi(\theta)| + |\chi(\theta)|)^q
		&\le [C_1(|\phi(\theta)| + |\chi(\theta)|)^\alpha + C_2]^q \\
		&\le 2^{q - 1} [C_1^q (|\phi(\theta)| + |\chi(\theta)|)^{\alpha q} + C_2^q],
	\end{align*}
the sequence $(g_n)_{n = 1}^\infty$ is dominated by the Lebesgue integrable function in view of
$\chi \in \mathcal{L}^{\alpha + 1}[-R, 0]$.
Thus, by applying Lebesgue's dominated convergence theorem, we have
	\begin{equation*}
		\lim_{n \to \infty}
		\left\| \frac{1}{h_n} [y(\cdot; \phi_0 + h_n\chi, r) - y(\cdot; \phi_0, r)] - B_{\phi, r}\chi \right\|_{C[-R, T]}
		= 0.
	\end{equation*}
This shows the G\^{a}teaux differentiability.

\textbf{Step 4. Continuous differentiability}

In view of the argument in Step 2, we have
	\begin{equation*}
		\|B_{\phi, r} - B_{\phi_0, r}\|
		\le \|Df \circ \phi - Df \circ \phi_0\|_{L^q[-R, 0]}
	\end{equation*}
for all $\phi, \phi_0 \in \mathcal{L}^{\alpha + 1}[-R, 0]$, where the convergence
	\begin{equation*}
		\|Df \circ \phi - Df \circ \phi_0\|_{L^q[-R, 0]} \to 0
			\mspace{20mu}
		\text{as $\|\phi - \phi_0\|_{L^{\alpha + 1}[-R, 0]} \to 0$}
	\end{equation*}
is a consequence of the continuity of composition operators
(Theorem~\ref{thm:continuity of composition operators in L^p}).
This shows the conclusion.
\end{proof}

When $Df \colon \R^N \to M_N(\R)$ is Lipschitz continuous,
another proof can be obtained as follows.

\begin{proof}[Another proof of Theorem~\ref{thm:smooth dependence, discontinuous history functional}
when $Df$ is Lipschitz continuous]
Let $r \in [T, R]$ be a parameter.
In this case, $\alpha$ can be taken as $1$.
Therefore, it is sufficient to consider the case $p = 2$.

\textbf{Step 1. Fr\'{e}chet differentiability}

Let $\phi \in \mathcal{L}^2[-R, 0]$ be given and $\chi \in \mathcal{L}^2[-R, 0]$.
Then we have
	\begin{align*}
		&\|y(\cdot; \phi + \chi, r) - y(\cdot; \phi, r) - B_{\phi, r}\chi\|_{C[-R, T]} \\
		&\le
		\int_{-R}^0
			|f((\phi + \chi)(\theta)) - f(\phi(\theta)) - Df(\phi(\theta))\chi(\theta)|
		\mspace{2mu} \mathrm{d}\theta \\
		&=
		\int_{-R}^0
			\left| \int_0^1 [Df((\phi + u\chi)(\theta)) - Df(\phi(\theta))]\chi(\theta) \mspace{2mu} \mathrm{d}u \right|
		\mspace{2mu} \mathrm{d}\theta.
	\end{align*}
Here for each $\theta \in [-R, 0]$,
	\begin{align*}
		&\left| \int_0^1 [Df((\phi + u\chi)(\theta)) - Df(\phi(\theta))]\chi(\theta) \mspace{2mu} \mathrm{d}u \right| \\
		&\le \int_0^1 \|Df((\phi + u\chi)(\theta)) - Df(\phi(\theta))\| |\chi(\theta)| \mspace{2mu} \mathrm{d}u \\
		&\le \frac{1}{2} \lip(Df) |\chi(\theta)|^2.
	\end{align*}
Since the last term is Lebesgue integrable in $\theta$,
	\begin{align*}
		\|y(\cdot; \phi + \chi, r) - y(\cdot; \phi, r) - B_{\phi, r}\chi\|_{C[-R, T]}
		&\le \frac{1}{2} \lip(Df) \int_{-R}^0 |\chi(\theta)|^2 \mspace{2mu} \mathrm{d}\theta \\
		&\le \frac{1}{2} \lip(Df) \|\chi\|_{L^2[-R, 0]} \|\chi\|_{L^2[-R, 0]} \\
		&= o(\|\chi\|_{L^2[-R, 0]})
	\end{align*}
as $\|\chi\|_{L^2[-R, 0]} \to 0$.
This shows the Fr\'{e}chet differentiability.

\textbf{Step 2. Continuous differentiability}

Let $\phi_1, \phi_2 \in \mathcal{L}^2[-R, 0]$.
In the same way as the Step 4 of the proof of Theorem~\ref{thm:smooth dependence, discontinuous history functional},
we have
	\begin{equation*}
		\|B_{\phi_1, r} - B_{\phi_2, r}\|
		\le \|Df \circ \phi_1 - Df \circ \phi_2\|_{L^2[-R, 0]}.
	\end{equation*}
Since $Df$ is Lipschitz continuous, the right-hand side is estimated as
	\begin{align*}
		\int_{-R}^0 \|Df(\phi_1(\theta)) - Df(\phi_1(\theta))\|^2 \mspace{2mu} \mathrm{d}\theta
		\le \lip(Df)^2 \|\phi_1 - \phi_2\|_{L^2[-R, 0]}^2.
	\end{align*}
Therefore,
	\begin{equation*}
		\|B_{\phi_1, r} - B_{\phi_2, r}\|
		\le \lip(Df)\|\phi_1 - \phi_2\|_{L^2[-R, 0]},
	\end{equation*}
which shows that $\mathcal{L}^2[-R, 0] \ni \phi \mapsto B_{\phi, r}$ is Lipschitz continuous.
\end{proof}

\begin{remark}
In the above proof, the continuity of the composition operator in $L^p$ is also unnecessary.
\end{remark}

\begin{corollary}\label{cor:smooth dependence, discontinuous history functional}
Let $0 < T < R$ be given.
Suppose that $f$ is of class $C^1$ and $\|Df(x)\| = O(|x|^\alpha)$ as $|x| \to \infty$ for some $\alpha \ge 1$.
For each $(\phi, r) \in \bar{\mathcal{L}}^p([-R, 0], \R^N) \times [T, R]$,
we define a linear map
	\begin{equation*}
		A_{\phi, r} \colon \bar{\mathcal{L}}^p([-R, 0], \R^N) \to \bar{\mathcal{L}}^p([-R, T], \R^N)
	\end{equation*}
by
	\begin{equation*}
		A_{\phi, r}\chi(t) =
		\begin{cases}
			\chi(t) & (t \in [-R, 0]), \\
			\chi(0) + \int_0^t Df(\phi(s - r))\chi(s - r) \mspace{2mu} \mathrm{d}s & (t \in [0, T]).
		\end{cases}
	\end{equation*}
Then for all $\alpha + 1 \le p < \infty$, $A_{\phi, r}$ is a bounded linear operator,
and $\phi \mapsto A_{\phi, r}$ is continuous with respect to the operator norm.
Furthermore, we have
	\begin{equation*}
		\|x(\cdot; \phi + \chi, r) - x(\cdot; \phi, r) - A_{\phi, r}\chi\|_{\bar{\mathcal{L}}^p[-R, T]}
		= o(\|\chi\|_{\bar{\mathcal{L}}^p[-R, 0]})
	\end{equation*}
as $\|\chi\|_{\bar{\mathcal{L}}^p[-R, 0]} \to 0$.
\end{corollary}

\begin{proof}
Let $(\phi, r) \in \bar{\mathcal{L}}^p[-R, 0] \times [T, R]$.

\textbf{Step 1. Boundedness of $A_{\phi, r}$}

The map $A_{\phi, r}$ is decomposed as
	\begin{equation*}
		A_{\phi, r} = (A_{\phi, r} - B_{\phi, r}) + B_{\phi, r}
	\end{equation*}
by using the map $B_{\phi, r}$ defined in Theorem~\ref{thm:smooth dependence, discontinuous history functional}.
Here we have the following properties:
\begin{itemize}
\item First term:
	\begin{equation*}
		A_{\phi, r} - B_{\phi, r} \colon \bar{\mathcal{L}}^p[-R, 0] \to \bar{\mathcal{L}}^p[-R, T]
	\end{equation*}
is the prolongation operator described in Lemma~\ref{lem:continuity of prolongation operator in L^p type history space}.
Therefore, it is continuous.
\item Second term: $B_{\phi, r} \colon \bar{\mathcal{L}}^p[-R, 0] \to C[-R, T]$ is a bounded linear operator
from Theorem~\ref{thm:smooth dependence, discontinuous history functional}.
Since the inclusion $C[-R, T] \subset \bar{\mathcal{L}}^p[-R, T]$ is continuous,
	\begin{equation*}
		B_{\phi, r} \colon \bar{\mathcal{L}}^p[-R, 0] \to \bar{\mathcal{L}}^p[-R, T]
	\end{equation*}
is also bounded.
\end{itemize}
By combining these properties, the boundedness of $A_{\phi, r}$ follows.

\textbf{Step 2. Small order estimate}

We use the decomposition
	\begin{equation*}
		x(\cdot; \phi, r) = y(\cdot; \phi, r) + \bar{\phi}.
	\end{equation*}
Then for all $\chi \in \bar{\mathcal{L}}^p[-R, 0]$, we have
	\begin{align*}
		&\|x(\cdot; \phi + \chi, r) - x(\cdot; \phi, r) - A_{\phi, r}\chi\|_{\bar{\mathcal{L}}^p[-R, T]} \\
		&= \|y(\cdot; \phi + \chi, r) - y(\cdot; \phi, r) - B_{\phi, r}\chi\|_{\bar{\mathcal{L}}^p[-R, T]} \\
		&\le (T + R + 1)^{\frac{1}{p}}\|y(\cdot; \phi + \chi, r) - y(\cdot; \phi, r) - B_{\phi, r}\chi\|_{C[-R, T]}.
	\end{align*}
Therefore, the estimate is a consequence of Theorem~\ref{thm:smooth dependence, discontinuous history functional}.

\textbf{Step 3. Continuous differentiability}

Let $q$ be the H\"{o}lder conjugate of $p$.
Let $\phi, \phi_0 \in \bar{\mathcal{L}}^p[-R, 0]$.
By the decomposition in Step 1, we have
	\begin{align*}
		\|(A_{\phi, r} - A_{\phi_0, r})\chi\|_{\bar{\mathcal{L}}^p[-R, T]}
		&= \|(B_{\phi, r} - B_{\phi_0, r})\chi\|_{\bar{\mathcal{L}}^p[-R, T]} \\
		&\le (T + R + 1)^{\frac{1}{p}} \|(B_{\phi, r} - B_{\phi_0, r})\chi\|_{C[-R, T]}
	\end{align*}
for all $\chi \in \bar{\mathcal{L}}^p[-R, 0]$.
From Theorem~\ref{thm:smooth dependence, discontinuous history functional}, this shows
	\begin{equation*}
		\|A_{\phi, r} - A_{\phi_0, r}\|
		\le (T + R + 1)^{\frac{1}{p}} \|Df \circ \phi - Df \circ \phi_0\|_{L^q[-R, 0]},
	\end{equation*}
where the right-hand side converges to $0$ as $\|\phi - \phi_0\|_{L^p[-R, 0]} \to 0$
by applying the continuity of composition operators in $L^p$
(Theorem~\ref{thm:continuity of composition operators in L^p}).
\end{proof}

\subsection{Regularity of solution semiflow and induced semiflow}

\begin{theorem}\label{thm:continuous semiflow, discontinuous history functional}
Let $0 < r \le R$.
If $|f(x)| = O(|x|^\alpha)$ as $|x| \to \infty$ for some $\alpha \ge 1$, then for all $\alpha \le p < \infty$,
the solution semiflow $\varPhi_r \colon \R_+ \times \bar{\mathcal{L}}^p([-R, 0], \R^N) \to \bar{\mathcal{L}}^p([-R, 0], \R^N)$
of \eqref{eq:single const delay I} given by
	\begin{equation*}
		\varPhi_r(t, \phi) := \varPhi_r^t(\phi) := R_tx(\cdot; \phi, r)
	\end{equation*}
is a continuous semiflow.
\end{theorem}

\begin{proof}
Let $T \in (0, r]$.
By the prolongability of $\bar{\mathcal{L}}^p[-R, 0]$
and the continuous maximal semiflow theorem (see Theorem~\ref{thm:continuity of maximal semiflow}),
we only have to show that the family $(\varPhi_r^t)_{t \in [0, T]}$ is equicontinuous
at each $\phi_0 \in \bar{\mathcal{L}}^p[-R, 0]$.
Let $t \in [0, T]$ and $\phi \in \bar{\mathcal{L}}^p[-R, 0]$.
By using the decomposition given by
	\begin{equation*}
		x(\cdot; \phi, r) = y(\cdot; \phi, r) + \bar{\phi},
	\end{equation*}
we have
	\begin{align*}
		&\|\varPhi_r(t, \phi) - \varPhi_r(t, \phi_0)\|_{\bar{\mathcal{L}}^p[-R, 0]} \\
		&= \|R_t[x(\cdot; \phi, r) - x(\cdot; \phi_0, r)]\|_{\bar{\mathcal{L}}^p[-R, 0]} \\
		&\le \|R_t[y(\cdot; \phi, r) - y(\cdot; \phi_0, r)]\|_{\bar{\mathcal{L}}^p[-R, 0]}
		+ \|R_t[\bar{\phi} - \bar{\phi}_0]\|_{\bar{\mathcal{L}}^p[-R, 0]}.
	\end{align*}

\textbf{First term.}
The regulation by prolongations of $\bar{\mathcal{L}}^p[-R, 0]$
(see Lemma~\ref{lem:regulation by prolongations of L^p type history space}) implies
	\begin{align*}
		\|R_t[y(\cdot; \phi, r) - y(\cdot; \phi_0, r)]\|_{\bar{\mathcal{L}}^p[-R, 0]}
		&\le (R + 1)^{\frac{1}{p}}\|R_t[y(\cdot; \phi, r) - y(\cdot; \phi_0, r)]\|_{C[-R, 0]} \\
		&\le (R + 1)^{\frac{1}{p}}\|y(\cdot; \phi, r) - y(\cdot; \phi_0, r)\|_{C[-R, T]},
	\end{align*}
where the last term converges to $0$ as $\|\phi - \phi_0\|_{L^p[-R, 0]} \to 0$
from the continuous dependence theorem
(Theorem~\ref{thm:continuous dependence, discontinuous history functional}).

\textbf{Second term.} The continuity of prolongation operator in $\bar{\mathcal{L}}^p$
(see Lemma~\ref{lem:continuity of prolongation operator in L^p type history space}) implies
	\begin{align*}
		\|R_t[\bar{\phi} - \bar{\phi}_0]\|_{\bar{\mathcal{L}}^p[-R, 0]}
		&\le \|\bar{\phi} - \bar{\phi}_0\|_{\bar{\mathcal{L}}^p[-R, t]} \\
		&\le (1 + t)^{\frac{1}{p}} \|\phi - \phi_0\|_{\bar{\mathcal{L}}^p[-R, 0]} \\
		&\le (1 + T)^{\frac{1}{p}}\|\phi - \phi_0\|_{\bar{\mathcal{L}}^p[-R, 0]} \\
		&\to 0
	\end{align*}
as $\|\phi - \phi_0\|_{\bar{\mathcal{L}}^p[-R, 0]} \to 0$.

This completes the proof.
\end{proof}

\begin{corollary}\label{cor:continuous semiflow, discontinuous history functional}
Let $0 < r \le R$.
If $|f(x)| = O(|x|^\alpha)$ as $|x| \to \infty$ for some $\alpha \ge 1$, then for all $\alpha \le p < \infty$,
the induced map
	\begin{gather*}
		\tilde{\varPhi}_r \colon \R_+ \times M^p([-R, 0], \R^N) \to M^p([-R, 0], \R^N), \\
		\tilde{\varPhi}_r(t, [\phi]) := \tilde{\varPhi}_r^t([\phi]) := [\varPhi_r(t, \phi)]
	\end{gather*}
is a continuous semiflow.
\end{corollary}

\begin{proof}
The semiflow property is checked as follows:
\begin{enumerate}
\item[(i)] For all $\phi \in \mathcal{L}^p[-R, 0]$, $\tilde{\varPhi}_r(0, [\phi]) = [\varPhi_r(0, \phi)] = [\phi]$.
\item[(ii)] For every $t, s \ge 0$ and every $\phi \in \mathcal{L}^p[-R, 0]$, we have
	\begin{align*}
		\tilde{\varPhi}_r(t + s, [\phi])
		&= [\varPhi_r(t + s, \phi)] \\
		&= [\varPhi_r(t, \varPhi_r(s, \phi))] \\
		&= \tilde{\varPhi}_r(t, \tilde{\varPhi}_r(s, [\phi])).
	\end{align*}
\end{enumerate}
The continuity is obtained in view of
	\begin{align*}
		\|\tilde{\varPhi}_r(t, [\phi]) - \tilde{\varPhi}_r(t_0, [\phi_0])\|
		&= \|\varPhi_r(t, \phi) - \varPhi_r(t_0, \phi_0)\|_{\bar{\mathcal{L}}^p[-R, 0]}, \\
		\|[\phi] - [\phi_0]\|
		&= \|\phi - \phi_0\|_{\bar{\mathcal{L}}^p[-R, 0]}.
	\end{align*}
This completes the proof.
\end{proof}

\begin{theorem}\label{thm:C^1-semiflow, discontinuous history functional}
Let $0 < r \le R$.
If $f$ is of class $C^1$ and $\|Df(x)\| = O(|x|^\alpha)$ as $|x| \to \infty$ for some $\alpha \ge 1$,
then for all $\alpha + 1 \le p < \infty$, the induced map
	\begin{gather*}
		\tilde{\varPhi}_r \colon \R_+ \times M^p([-R, 0], \R^N) \to M^p([-R, 0], \R^N), \\
		\tilde{\varPhi}_r(t, [\phi]) := \tilde{\varPhi}_r^t([\phi]) := [\varPhi_r(t, \phi)]
	\end{gather*}
is a semiflow of class $C^1$.
\end{theorem}

\begin{proof}
Let $0 < T \le r$.
The assumption implies that $|f(x)| = O(|x|^{\alpha + 1})$ as $|x| \to \infty$.
Then, Corollary~\ref{cor:continuous semiflow, discontinuous history functional} states that
$\tilde{\varPhi}_r$ is a continuous semiflow.
Therefore, we only have to show that for each $t \in [0, T]$, $\tilde{\varPhi}_r^t$ is of class $C^1$
by the $C^1$-maximal semiflow theorem (see Theorem~\ref{thm:C^1-maximal semiflow theorem}).

Let $t \in [0, T]$ and $\phi \in \mathcal{L}^p[-R, 0]$.
We use the linear maps $A_{\phi, r}$ and $B_{\phi, r}$ introduced in the smooth dependence theorems
(see Corollary~\ref{cor:smooth dependence, discontinuous history functional}
and Theorem~\ref{thm:smooth dependence, discontinuous history functional})
and the decomposition given by
	\begin{equation*}
		x(\cdot; \phi, r) = y(\cdot; \phi, r) + \bar{\phi}.
	\end{equation*}

\textbf{Step 1. Fr\'{e}chet differentiability}

Let $R_t\tilde{A}_{\phi, r} \colon M^p[-R, 0] \to M^p[-R, 0]$ be the linear map defined by
	\begin{equation*}
		R_t\tilde{A}_{\phi, r}[\chi] = [R_t(A_{\phi, r}\chi)].
	\end{equation*}
Let $\chi \in \mathcal{L}^p[-R, 0]$.
Then we have
	\begin{align*}
		\|R_t\tilde{A}_{\phi, r}[\chi]\|
		&= \|R_t(A_{\phi, r}\chi)\|_{\bar{\mathcal{L}}^p[-R, 0]} \\
		&\le \|R_t(A_{\phi, r} - B_{\phi, r})\chi\|_{\bar{\mathcal{L}}^p[-R, 0]}
		+ \|R_tB_{\phi, r}\chi\|_{\bar{\mathcal{L}}^p[-R, 0]} \\
		&\le \|\bar{\chi}\|_{\bar{\mathcal{L}}^p[-R, T]} + (R + 1)^{\frac{1}{p}} \|B_{\phi, r}\chi\|_{C[-R, T]} \\
		&\le \bigl[ (1 + T)^{\frac{1}{p}} + (R + 1)^{\frac{1}{p}} \bigr] \|[\chi]\|.
	\end{align*}
This shows that $R_t\tilde{A}_{\phi, r} \colon M^p[-R, 0] \to M^p[-R, 0]$ is bounded.

The Fr\'{e}chet differentiability is obtained because
	\begin{align*}
		&\|\tilde{\varPhi}_r(t, [\phi] + [\chi]) - \tilde{\varPhi}_r(t, [\phi]) - R_t\tilde{A}_{\phi, r}[\chi]\| \\
		&= \|\varPhi_r(t, \phi + \chi) - \varPhi_r(t, \phi) - R_t(A_{\phi, r}\chi)\|_{\bar{\mathcal{L}}^p[-R, 0]} \\
		&= \|R_t(y(\cdot; \phi + \chi, r) - y(\cdot; \phi, r) - B_{\phi, r}\chi)\|_{\bar{\mathcal{L}}^p[-R, 0]} \\
		&\le (R + 1)^{\frac{1}{p}}\|y(\cdot; \phi + \chi, r) - y(\cdot; \phi, r) - B_{\phi, r}\chi\|_{C[-R, T]} \\
		&= o(\|\chi\|_{L^p[-R, 0]}) \\
		&= o(\|[\chi]\|)
	\end{align*}
as $\|[\chi]\| \to 0$.

\textbf{Step 2. Continuity of the derivative}

Let $\phi_0 \in \mathcal{L}^p[-R, 0]$.
For all $\phi \in \mathcal{L}^p[-R, 0]$, we have
	\begin{align*}
		\|(R_t\tilde{A}_{\phi, r} - R_t\tilde{A}_{\phi_0, r})[\chi]\|
		&= \|R_t(B_{\phi, r}\chi - B_{\phi_0, r}\chi)\|_{\bar{\mathcal{L}}^p[-R, 0]} \\
		&\le (R + 1)^{\frac{1}{p}} \|B_{\phi, r} - B_{\phi_0, r}\| \|[\chi]\|.
	\end{align*}
This shows
	\begin{equation*}
		\|R_t\tilde{A}_{\phi, r} - R_t\tilde{A}_{\phi_0, r}\|
		\le \|B_{\phi, r} - B_{\phi_0, r}\|,
	\end{equation*}
where the right-hand side converges to $0$ as $\|[\phi] - [\phi_0]\| = \|\phi - \phi_0\|_{\bar{\mathcal{L}}^p[-R, 0]} \to 0$.
Therefore, $M^p[-R, 0] \ni \phi \mapsto R_t\tilde{A}_{\phi, r}$ is continuous.

This completes the proof.
\end{proof}

\section{Comments and discussion}

This paper studies a special form of delay differential equations
as a retarded functional differential equation with a discontinuous history functional and discontinuous initial histories.
From this study, it becomes clear that there is a possibility of obtaining the better smooth dependence of solution
on initial conditions even if the history functional does not have nice regularity.

By restricting the form of delay differential equations,
we can clarify the connection between this smooth dependence result and the smoothness of the composition operator.
It is natural to investigate more general form of delay differential equations as a next task,
which is also motivated by the Galerkin approximation of delay differential equations
studied by Chekroun et al.~\cite{Chekroun--Ghil--Liu--Wang 2016} and
Chekroun, Kr\"{o}ner, \& Liu~\cite{Chekroun--Kroner--Liu preprint}.
%\section*{Acknowledgment}

\appendix
\section{Continuity and smoothness of composition operators in $L^p$}
\label{sec:composition operators in L^p}

We need the following generalized version of Lebesgue's dominated convergence theorem.
The proof can be omitted because the method of modification is obvious.
We refer the reader to \cite{Tao 2011} as a general theory of Lebesgue integration.

\begin{theorem}[Generalized Lebesgue dominated convergence theorem]
\label{thm:generalized Lebesgue DCT}
Let $X = (X, \mu)$ be a measure space.
Let $f_n \colon X \to \R$ be a measurable function for each integer $n \ge 1$.
Suppose that the sequence $(f_n)_{n = 1}^\infty$ converges pointwise almost everywhere
to a measurable function $f \colon X \to \R$.
If there exist a sequence $(g_n)_{n = 1}^\infty$ in $\mathcal{L}^1(X, \R)$ and $g \in \mathcal{L}^1(X, \R)$ such that
\begin{itemize}
\item $|f_n| \le g_n$ holds almost everywhere for all $n \ge 1$,
\item $(g_n)_{n = 1}^\infty$ converges pointwise almost everywhere to $g$,
\item $\lim_{n \to \infty} \int_X g_n \mspace{2mu} \mathrm{d}\mu = \int_X g \mspace{2mu} \mathrm{d}\mu$,
\end{itemize}
then all the functions $f_n$ and $f$ are Lebesgue integrable and
	\begin{equation*}
		\lim_{n \to \infty} \int_X f_n \mspace{2mu} \mathrm{d}\mu
		= \int_X f \mspace{2mu} \mathrm{d}\mu
	\end{equation*}
holds.
\end{theorem}

As a corollary of the above generalized convergence theorem,
we obtain the following convergence theorem in $L^p$.

\begin{theorem}[Dominated convergence in $L^p$]\label{thm:dominated convergence in L^p}
Let $X = (X, \mu)$ be a measure space.
Let $1 \le p < \infty$ and $f_n \colon X \to \R^N$ be a measurable function for each integer $n \ge 1$.
Suppose that the sequence $(f_n)_{n = 1}^\infty$ converges pointwise almost everywhere
to a measurable function $f \colon X \to \R^N$.
If there exist a sequence $(g_n)_{n = 1}^\infty$ in $\mathcal{L}^1(X, \R)$ and $g \in \mathcal{L}^1(X, \R)$ such that
\begin{itemize}
\item $|f_n|^p \le g_n$ holds almost everywhere for all $n \ge 1$,
\item $(g_n)_{n = 1}^\infty$ converges pointwise almost everywhere to $g$,
\item $\lim_{n \to \infty} \int_X g_n \mspace{2mu} \mathrm{d}\mu = \int_X g \mspace{2mu} \mathrm{d}\mu$,
\end{itemize}
then $f_n, f \in \mathcal{L}^p(X, \R^N)$ for all $n \ge 1$, and
$\lim_{n \to \infty} \|f_n - f\|_{L^p(X)}$ holds.
\end{theorem}

\begin{theorem}[Continuity of composition operators in $L^p$]
\label{thm:continuity of composition operators in L^p}
Let $X = (X, \mu)$ be a finite measure space and $p, q \in [1, \infty)$.
Let $M, N \ge 1$ be integers and $f \colon \R^M \to \R^N$ be a continuous function.
Suppose $|f(y)| = O(|y|^\alpha)$ as $|y| \to \infty$ for some $\alpha \ge 1$.
If $p = \alpha q$, then the composition operator
	\begin{equation*}
		\mathcal{L}^p(X, \R^M) \ni g \mapsto f \circ g \in \mathcal{L}^q(X, \R^N)
	\end{equation*}
is a well-defined continuous map.
\end{theorem}

\begin{proof}
By the assumption, we choose $C_1, C_2 > 0$ so that
	\begin{equation*}
		|f(y)| \le C_1|y|^{\alpha} + C_2 \mspace{20mu} (\forall y \in \R^M).
	\end{equation*}

\textbf{Step 1. Well-definedness}

For all $g \in \mathcal{L}^p(X, \R^M)$, we have
	\begin{align*}
		|f \circ g|^q
		\le (C_1|g|^\alpha + C_2)^q
		\le 2^{q - 1}(C_1^q|g|^{\alpha q} + C_2^q),
	\end{align*}
which implies $f \circ g \in \mathcal{L}^q(X, \R^N)$.

\textbf{Step 2. Continuity}

Let $(g_n)_{n = 1}^\infty$ be a sequence in $\mathcal{L}^p(X, \R^M)$ which converges to $g$.
We will show that
	\begin{equation*}
		I_n := \|f \circ g_n - f \circ g\|_{L^q(X)} \mspace{20mu} (n \ge 1)
	\end{equation*}
converges to $0$ as $n \to \infty$.
Let $(I_{n_k})_{k = 1}^\infty$ be a subsequence.
By the assumption, $g_{n_k} \to g$ in measure as $k \to \infty$.
Therefore, there is a subsequence $(g_{n'_\ell})_{\ell = 1}^\infty$ of $(g_{n_k})_{k = 1}^\infty$
such that $g_{n'_\ell}$ converges pointwise almost everywhere to $g$ as $\ell \to \infty$.
Then $f \circ g_{n'_\ell}$ also converges pointwise almost everywhere to $f \circ g$ as $\ell \to \infty$
by the continuity of $f$.

Let $h_\ell, h \colon X \to \R$ be the Lebesgue integrable functions defined by
	\begin{equation*}
		h_\ell = 2^{q - 1}(C_1^q|g_{n'_\ell}|^{\alpha q} + C_2^q),
			\mspace{15mu}
		h = 2^{q - 1}(C_1^q|g|^{\alpha q} + C_2^q)
	\end{equation*}
for each $\ell \ge 1$.
Then the following properties hold:
\begin{itemize}
\item $|f \circ g_{n'_\ell}|^q \le h_\ell$ holds for all $\ell \ge 1$.
\item $(h_\ell)_{\ell = 1}^\infty$ converges pointwise almost everywhere to $h$.
\item $\lim_{\ell \to \infty} \int_X h_\ell \mspace{2mu} \mathrm{d}\mu = \int_X h \mspace{2mu} \mathrm{d}\mu$ holds
because
	\begin{equation*}
		\int_X h_\ell \mspace{2mu} \mathrm{d}\mu
		= 2^{q - 1} \left( C_1^q \int_X |g_{n'_\ell}|^{\alpha q}\mathrm{d}\mu + C_2^q\mu(X) \right).
	\end{equation*}
\end{itemize}
Therefore, the dominated convergence in $L^q$ (Theorem~\ref{thm:dominated convergence in L^p}) implies
$\lim_{\ell \to \infty} I_{n'_\ell} = 0$.
This means that each subsequence $(I_{n_k})_{k = 1}^\infty$ has an accumulation point $0$,
which is independent from the choice of the subsequence.
Thus, $\lim_{n \to \infty} I_n = 0$ is obtained.
\end{proof}

\begin{remark}
See \cite[Theorem 3.6]{Boccardo--Croce 2013} for another proof.
\end{remark}

\begin{theorem}[Smoothness of composition operators in $L^p$]
\label{thm:smoothness of composition operators in L^p}
Let $X = (X, \mu)$ be a finite measure space and $p, q \in [1, \infty)$.
Let $M, N \ge 1$ be integers and $f \colon \R^M \to \R^N$ be a function of class $C^1$.
Suppose $\|Df(y)\| = O(|y|^\alpha)$ as $|y| \to \infty$ for some $\alpha \ge 1$.
If $p = (\alpha + 1)q$, then the composition operator
	\begin{equation*}
		T \colon L^p(X, \R^M) \ni g \mapsto f \circ g \in L^q(X, \R^N)
	\end{equation*}
is a well-defined continuously Fr\'{e}chet differentiable map,
whose Fr\'{e}chet derivative is given by
	\begin{equation*}
		[DT(g)h](x) = Df(g(x))h(x) \mspace{20mu} \text{almost every $x \in X$}
	\end{equation*}
for all $h \in L^p(X, \R^M)$.
\end{theorem}

\begin{proof}
\textbf{Step 1. Well-definedness}

By the assumption, $|f(y)| = O(|y|^{\alpha + 1})$ as $|y| \to \infty$.
Therefore, the map $T$ is well-defined from Theorem~\ref{thm:continuity of composition operators in L^p}.

\textbf{Step 2. Order estimate of the Lipschitz constant}

We choose $C_1, C_2 > 0$ so that
	\begin{equation*}
		\|Df(y)\| \le C_1|y|^\alpha + C_2 \mspace{20mu} (\forall y \in \R^M).
	\end{equation*}
Let $r > 0$.
Then for all $y_1, y_2 \in \bar{B}(0; r)$ (the closed ball with center $0$ and radius $r$), we have
	\begin{equation*}
		|f(y_1) - f(y_2)| \le \sup_{y \in \bar{B}(0; r)} \|Df(y)\| \cdot |y_1 - y_2|.
	\end{equation*}
This implies
	\begin{equation*}
		\lip \bigl( f|_{\bar{B}(0; r)} \bigr) \le C_1r^\alpha + C_2 =: L(r).
	\end{equation*}

\textbf{Step 3. G\^{a}teaux differentiability}

For given $g, h \in L^p(X, \R^M)$, let
	\begin{equation*}
		A_gh(x) := Df(g(x))h(x)
	\end{equation*}
for almost all $x \in X$.
Let $(t_n)_{n = 1}^\infty$ be a sequence in $[-1, 1] \setminus \{0\}$ which converges to $0$.
Then
	\begin{align*}
		&\left\| \frac{1}{t_n} [T(g + t_nh) - T(g)] - A_gh \right\|_{L^q(X)}^q \\
		&\le
		\int_X
			\left| \frac{1}{t_n} [f(g(x) + t_nh(x)) - f(g(x))] - Df(g(x))h(x) \right|^q
		\mspace{2mu} \mathrm{d}\mu(x).
	\end{align*}
We define functions $F_n, F \colon X \to \R^N$ by
	\begin{equation*}
		F_n(x) = \frac{1}{t_n} [f(g(x) + t_nh(x)) - f(g(x))],
			\mspace{15mu}
		F(x) = Df(g(x))h(x).
	\end{equation*}
Then for almost all $x \in X$, $F_n(x) \to F(x)$ as $n \to \infty$, and
	\begin{align*}
		|F_n(x)|
		&= \frac{1}{|t_n|} |f(g(x) + t_nh(x)) - f(g(x))| \\
		&\le \frac{1}{|t_n|} L(|g(x)| + |h(x)|) |t_nh(x)| \\
		&= L(|g(x)| + |h(x)|) |h(x)|.
	\end{align*}
Since
	\begin{align*}
		&\int_X L(|g(x)| + |h(x)|)^q |h(x)|^q \mspace{2mu} \mathrm{d}\mu(x) \\
		&\le \left( \int_X L(|g(x)| + |h(x)|)^{(\alpha + 1)q/\alpha} \mspace{2mu} \mathrm{d}\mu(x) \right)^{\alpha/(\alpha + 1)}
		\left( \int_X |h|^{(\alpha + 1)q} \mspace{2mu} \mathrm{d}\mu \right)^{1/(\alpha + 1)}
	\end{align*}
where
	\begin{align*}
		L(|g(x)| + |h(x)|)^{(\alpha + 1)q/\alpha}
		&\le [C_1(|g(x)| + |h(x)|)^\alpha + C_2]^{(\alpha + 1)q/\alpha} \\
		&\le 2^{(p/\alpha) - 1} [C_1^{p/\alpha} (|g(x)| + |h(x)|)^p + C_2^{p/\alpha}],
	\end{align*}
the sequence $(|F_n|^q)_{n = 1}^\infty$ is dominated by the Lebesgue integrable function.
Thus, by applying the dominated convergence in $L^q$ (Theorem~\ref{thm:dominated convergence in L^p}), we have
	\begin{equation*}
		\lim_{n \to \infty}
		\left\| \frac{1}{t_n} [T(g + t_nh) - T(g)] - A_gh \right\|_{L^q(X)}^q
		= 0.
	\end{equation*}
This shows the G\^{a}teaux differentiability.

\textbf{Step 4. Boundedness of G\^{a}teaux derivative}

Let $g \in L^p(X, \R^M)$.
Then for all $h \in L^p(X, \R^M)$, we have
	\begin{align*}
		&\int_X |A_gh(x)|^q \mspace{2mu} \mathrm{d}\mu(x) \\
		&\le \int_X \|Df(g(x))\|^q |h(x)|^q \mspace{2mu} \mathrm{d}\mu(x) \\
		&\le \left( \int_X \|Df \circ g\|^{(\alpha + 1)q/\alpha} \mspace{2mu} \mathrm{d}\mu \right)^{\alpha/(\alpha + 1)}
		\left( \int_X |h|^{(\alpha + 1)q} \mspace{2mu} \mathrm{d}\mu \right)^{1/(\alpha + 1)}.
	\end{align*}
Since
	\begin{align*}
		\|Df(g(x))\|^{p/\alpha}
		&\le (C_1|g(x)|^\alpha + C_2)^{p/\alpha} \\
		&\le 2^{(p/\alpha) - 1}(C_1^{p/\alpha}|g(x)|^p + C_2^{p/\alpha})
	\end{align*}
for almost all $x \in X$, $\|Df \circ g\|_{L^{p/\alpha}(X)} < \infty$.
Therefore, the linear map $A_g \colon L^p(X, \R^M) \to L^q(X, \R^N)$ is well-defined,
and the operator norm is estimated as
	\begin{equation*}
		\|A_g\| \le \|Df \circ g\|_{L^{p/\alpha}(X)}.
	\end{equation*}

\textbf{Step 5. Continuous Fr\'{e}chet differentiability}

In view of the argument in Step 4, we have
	\begin{equation*}
		\|A_g - A_{g_0}\| \le \|Df \circ g - Df \circ g_0\|_{L^{p/\alpha}(X)}
	\end{equation*}
for all $g, g_0 \in L^p(X, \R^M)$, where the convergence
	\begin{equation*}
		\|Df \circ g - Df \circ g_0\|_{L^{p/\alpha}(X)} \to 0
		\mspace{20mu}
		\text{as $\|g - g_0\|_{L^p(X)} \to 0$}
	\end{equation*}
is a consequence of Theorem~\ref{thm:continuity of composition operators in L^p}
because $p = \alpha \cdot (p/\alpha)$.
This shows that $T$ is continuously Fr\'{e}chet differentiable.

This completes the proof.
\end{proof}

\begin{remark}
See \cite[Theorem 7]{Goldberg--Kampowsky--Troltzsch 1992} for another proof.
\end{remark}

\section{Regularity of maximal semiflows}\label{sec:regularity of maximal semiflows}

\begin{definition}[Maximal semiflow]
Let $X$ be a set and $D \subset \R_+ \times X$ be a subset.
A map $\varPhi \colon D \to X$ is called a \textit{maximal semiflow} in $X$
if the following conditions are satisfied:
\begin{enumerate}
\item[(i)] There exists a function $T_\varPhi \colon X \to (0, \infty]$ such that
	\begin{equation*}
		D = \bigcup_{x \in X} \bigl( [0, T_\varPhi(x)) \times \{x\} \bigr).
	\end{equation*}
\item[(ii)] For all $x \in X$, $\varPhi(0, x) = x$.
\item[(iii)] For each $x \in X$, the following statement holds:
For all $t, s \in \R_+$, $(t, x) \in D$ and $(s, \varPhi(t, x)) \in D$ imply
	\begin{equation*}
		(t + s, x) \in D, \mspace{20mu} \varPhi(t + s, x) = \varPhi(s, \varPhi(t, x)).
	\end{equation*}
\end{enumerate}
The function $T_\varPhi$ is called the \textit{escape time function}.
\end{definition}

\begin{definition}[Continuous maximal semiflow]
Let $X$ be a topological space.
A maximal semiflow $\varPhi$ in $X$ is called a \textit{continuous maximal semiflow}
if the map $\varPhi$ is continuous and the escape time function $T_\varPhi \colon X \to (0, \infty]$ is lower semicontinuous.
\end{definition}

\begin{definition}[Maximal semiflow of class $C^1$]
Let $X$ be a normed space.
A maximal semiflow $\varPhi$ in $X$ is called a \textit{maximal semiflow of class $C^1$}
if $\varPhi$ is a continuous maximal semiflow and each time-$t$ map $\varPhi^t$ is continuously Fr\'{e}chet differentiable.
\end{definition}

\begin{theorem}\label{thm:continuity of maximal semiflow}
Let $X$ be a topological space and $\varPhi \colon D \to X$ be a maximal semiflow in $X$
with the escape time function $T_\varPhi \colon X \to (0, \infty]$.
Suppose that for every $x \in X$, the orbit $[0, T_\varPhi(x)) \ni t \mapsto \varPhi(t, x) \in X$ is continuous.
If for every $x \in X$, there exist $T > 0$ and a neighborhood $N$ of $x$ in $X$ such that
(i) $[0, T] \times N \subset D$ and (ii) $\varPhi|_{[0, T] \times N}$ is continuous,
then $\varPhi$ is a continuous maximal semiflow.
\end{theorem}

\begin{remark}
The statement in Theorem~\ref{thm:continuity of maximal semiflow} is proved
by H\'{a}jek~\cite[Theorem 15]{Hajek 1968} with a weaker assumption.
See also the proof of \cite[Theorem A.7]{Nishiguchi 2017}.
\end{remark}

\begin{theorem}\label{thm:C^1-maximal semiflow theorem}
Let $X$ be a normed space and $\varPhi \colon D \to X$ be a continuous maximal semiflow in $X$
with the escape time function $T_\varPhi \colon X \to (0, \infty]$.
If for every $x \in X$, there exist $T > 0$ and an open neighborhood $N$ of $x$ such that
(i) $[0, T] \times N \subset D$ and (ii) $\varPhi^t|_N$ is of class $C^1$ for each $t \in [0, T]$,
then $\varPhi$ is a maximal semiflow of class $C^1$.
\end{theorem}

\begin{remark}
The proof of Theorem~\ref{thm:C^1-maximal semiflow theorem} is similar to
that of Theorem~\ref{thm:continuity of maximal semiflow}.
See also the proof of \cite[Theorem 1]{Walther 2003c}.
\end{remark}


\begin{thebibliography}{99}
\addcontentsline{toc}{section}{References}

\bibitem{Boccardo--Croce 2013}
	\newblock L. Boccardo and G. Croce,
	\newblock ``Elliptic Partial Differential Equations: Existence and Regularity of Distributional Solutions,"
	\newblock Walter de Gruyter, 2013.

\bibitem{Breda 2010} %(MR2675588)
	\newblock D. Breda,
	\newblock \textit{Nonautonomous delay differential equations in Hilbert spaces and Lyapunov exponents},
	\newblock Differential Integral Equations \textbf{23} (2010), 935--956.

\bibitem{Burns--Herdman--Stech 1983} %(MR0686237)
	\newblock J. A. Burns, T. L. Herdman and H. W. Stech,
	\newblock \textit{Linear functional-differential equations as semigroups on product spaces},
	\newblock SIAM J. Math. Anal. \textbf{14} (1983), 98--116.

\bibitem{Chekroun--Ghil--Liu--Wang 2016} %(MR3479510)
	\newblock M. D. Chekroun, M. Ghil, H. Liu and S. Wang,
	\newblock \textit{Low-dimensional Galerkin approximations of nonlinear delay differential equations},
	\newblock Discrete Contin. Dyn. Syst. \textbf{36} (2016), 4133--4177.

\bibitem{Chekroun--Kroner--Liu preprint}
	\newblock M . D. Chekroun, A. Kr\"{o}ner and H. Liu,
	\newblock \textit{Galerkin approximations for the optimal control of nonlinear delay differential equations},
	\newblock arXiv preprint arXiv:1706.02360.

\bibitem{Delfour 1980} %(MR0590659)
	\newblock M. C. Delfour,
	\newblock \textit{The largest class of hereditary systems defining a $C_0$ semigroup on the product space},
	\newblock Canad. J. Math. \textbf{32} (1980), 969--978.

\bibitem{Delfour--Mitter 1972c} %(MR0328261)
	\newblock M. C. Delfour and S. K. Mitter,
	\newblock \textit{Hereditary differential systems with constant delays. I. General case},
	\newblock J. Differential Equations \textbf{12} (1972), 213--235.

\bibitem{Goldberg--Kampowsky--Troltzsch 1992} %(MR1231260)
	\newblock H. Goldberg, W. Kampowsky and F. Tr\"{o}ltzsch,
	\newblock \textit{On Nemytskij operators in $L_p$-spaces of abstract functions},
	\newblock Math. Nachr. \textbf{155} (1992), 127--140.

\bibitem{Hajek 1968} %(MR0239575)
	\newblock O. H\'{a}jek,
	\newblock \textit{Local characterisation of local semi-dynamical systems},
	\newblock Math. Systems Theory \textbf{2} (1968), 17--25.

\bibitem{Hale 1963b} %(MR0157064)
	\newblock J. K. Hale,
	\newblock \textit{A stability theorem for functional-differential equations},
	\newblock Proc. Nat. Acad. Sci. U.S.A. \textbf{50} (1963), 942--946.

\bibitem{Hale 1969} %(MR0244582)
	\newblock J. K. Hale,
	\newblock \textit{Dynamical systems and stability},
	\newblock J. Math. Anal. Appl. \textbf{26} (1969), 39--59.

\bibitem{Hale--Lunel 1993} %(MR1243878)
	\newblock J. K. Hale and S. M. Verduyn Lunel,
	\newblock ``Introduction to Functional Differential Equations,''
	\newblock Springer-Verlag, New York, 1993.

\bibitem{Herdman--Burns 1979} %(MR0551033)
	\newblock T. L. Herdman and J. A. Burns,
	\newblock \textit{Functional differential equations with discontinuous right-hand side},
	\newblock Volterra equations, 99--106, Springer, Berlin, 1979.

\bibitem{Hino--Murakami--Naito 1991} %(MR1122588)
	\newblock Y. Hino, S. Murakami, and T. Naito,
	\newblock ``Functional-differential Equations with Infinite Delay," 
	\newblock Springer-Verlag, Berlin, 1991.

\bibitem{Kappel--Schappacher 1978} %(MR0512478)
	\newblock F. Kappel and W. Schappacher,
	\newblock \textit{Autonomous nonlinear functional differential equations and averaging approximations},
	\newblock Nonlinear Anal. \textbf{2} (1978), 391--422.

\bibitem{Nishiguchi 2017}
	\newblock J. Nishiguchi,
	\newblock \textit{A necessary and sufficient condition for well-posedness of initial value problems of retarded functional differential equations},
	\newblock J. Differential Equations \textbf{263} (2017), 3491--3532.

\bibitem{Nishiguchi preprint}
	\newblock J. Nishiguchi,
	\newblock \textit{Theory of well-posedness for delay differential equations via prolongations and $C^1$-prolongations:
	its application to state-dependent delay},
	\newblock submitted (arXiv:1810.05890).

\bibitem{Tao 2011} %(MR2827917)
	\newblock T. Tao,
	\newblock ``An introduction to Measure Theory,"
	\newblock American Mathematical Society, Providence, RI, 2011.

\bibitem{Walther 2003c} %(MR2019242)
	\newblock H.-O. Walther,
	\newblock \textit{The solution manifold and $C^1$-smoothness for differential equations with state-dependent delay},
	\newblock J. Differential Equations \textbf{195} (2003), 46--65.

\bibitem{Webb 1976} %(MR0390422)
	\newblock G. F. Webb,
	\newblock \textit{Functional differential equations and nonlinear semigroups in $L^p$-spaces},
	\newblock J. Differential Equations \textbf{20} (1976), 71--89.

\end{thebibliography}
\end{document}